\documentclass[reqno]{amsart}

\usepackage{amsfonts,amsthm,amsmath,amssymb,amscd,amsmath,epsf,latexsym,mathrsfs}
\usepackage[dvipsnames]{xcolor}
\usepackage{graphicx}
\usepackage{url}
\newtheorem{theorem}{Theorem}
\newtheorem{proposition}[theorem]{Proposition}
\newtheorem{lemma}[theorem]{Lemma}

\newtheorem{corollary}{Corollary}
\newtheorem{example}{Example}
\newtheorem{remark}{Remark}
\newtheorem{cor}{Corollary}

\newcommand{\bR}{\mathbb{R}}
\newcommand{\bC}{\mathbb{C}}
\newcommand{\ee}{\end{equation}}

\newcommand{\D}{\mathcal D}

\newcommand \Vor {\rm{\text{Vor}}}
\newcommand*\diff{\mathop{}\!\mathrm{d}}

\begin{document}
          \numberwithin{equation}{section}
           \numberwithin{theorem}{section}
          
              \title[P\'olya's method to construct Voronoi diagrams]
          {A refinement for rational functions of P\'olya's method to construct Voronoi diagrams}

 \author[R.~B\o gvad]{Rikard B\o gvad}
\address{Department of Mathematics, Stockholm University, SE-106 91
Stockholm,         Sweden}
\email {rikard@math.su.se }

\author[Ch.~H\"agg]{Christian H\"agg}
\address{Department of Mathematics, Stockholm University, SE-106 91
Stockholm,         Sweden}
\email{hagg@math.su.se}

\begin{abstract} Given a complex polynomial $P$ with zeroes $z_1,\dotsc,z_d$, we show that the asymptotic zero-counting measure of the iterated derivatives $Q^{(n)}, \ n=1,2,\dotsc$, where $Q=R/P$ is any irreducible rational function, converges to an explicitly constructed probability measure supported by the Voronoi diagram associated with $z_1,\dotsc,z_d$. This refines P\'olya's Shire theorem for these functions. In addition, we prove a similar result, using currents, for Voronoi diagrams associated with generic hyperplane configurations in $\bC^m$. 
\end{abstract}

\maketitle

\section{Introduction}

P\'olya's Shire theorem \cite{Ha, Po1,Po2, Whit} says that the zero sets $Z(f^{(n)})$ of the iterated derivatives $f^{(n)}$ of a meromorphic function $f$ with set of poles $S$ accumulate along (the boundaries of) the Voronoi diagram associated with $S$. Considering the many recent studies of weak limits of zero-set measures of polynomial sequences, it is tempting to use that circle of ideas in the situation of P\'olya's theorem. In this note we will do this for rational functions, where a measure-theoretic formulation is rather immediate, and the proof is straightforward.

If $Q=R/P$ (where we always assume that $\gcd(R,P)=1$) is a rational function, the {\it associated zero-counting measure} $\mu_{_Q}$ is defined as the discrete probability measure that assigns equal weight to all zeroes $a_1,a_2,\dotsc,a_n$ of $Q$ (counted with multiplicity). That is 
$$\mu_{_Q}=\frac{1}{n}\sum_{i=1}^n \delta_{a_i},
$$ where $\delta_{a_i}$ is the Dirac measure at $a_i$. Fix a rational function $Q=R/P$. The Voronoi diagram $\Vor_{_S}$ associated with the zeroes $S:=\{z_1,\dotsc,z_d\}$ of $P$, consists of (certain, see below) segments on the lines $\vert z-z_i\vert=\vert z-z_j\vert$, where $z_i,\,z_j$ are distinct zeroes of $P$.
Our main result is a description of the asymptotic limit of the zero-counting measures of the derivatives $Q_n := Q^{(n)}$.
Define a plane measure with support on the line $L_{ij}: \vert z-z_i\vert=\vert z-z_j\vert$ by
 $$\mu_{ij} := \frac{1}{2(d-1)\pi}\frac{\vert z_i - z_j\vert}{\vert (z-z_i)(z-z_j)\vert}\diff s,$$
 where $s$ is Euclidean length measure in the complex plane, and $z_i,\,z_j$ are distinct zeroes of $P$. Restricting the measure to the segment of $L_{ij}$ that is part of the Voronoi diagram, and summing over all lines gives a measure $\mu_{_S}$,
  supported on the Voronoi diagram. This will in fact be a probability measure canonically associated with the diagram.

\begin{theorem}\label{th:voronoi}
Given a rational function $Q=R/P$ where $P$ has degree  $d\ge 2$ and distinct zeroes $z_1,\dotsc,z_d$, \\
\noindent(i) the zero-counting measures $\mu_n$ of the sequence $\{Q_n\}_{n=1}^\infty$ converge to the probability measure $\mu_{_S}$.

\noindent(ii)  The logarithmic potential $L_{\mu_n}(z)$ of $\mu_n$ converges in $L^1_{loc}$ to the logarithmic potential of $\mu_{_S}$, which equals
$$\Psi(z): = (d-1)^{-1}(\log |\tilde P| + \smash{\displaystyle\max_{i=1,\dotsc,d}} \{\log |z-z_i|^{-1}\}),$$
where $\tilde P=\prod_{i=1}^d(z-z_i)$.
\end{theorem}

\begin{figure}
\begin{center}
\includegraphics[scale=0.50]{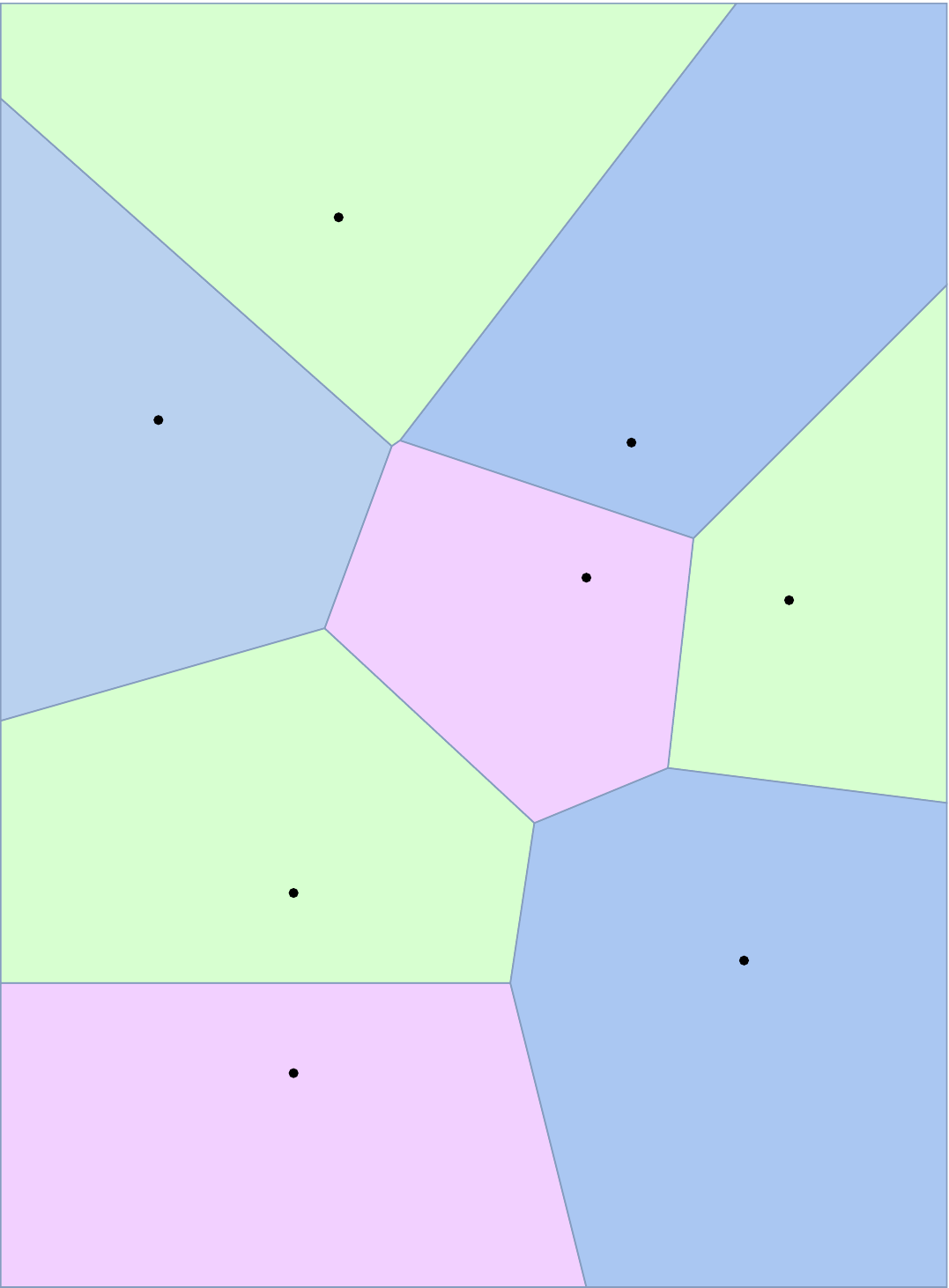}\quad \quad \includegraphics [scale=0.50]{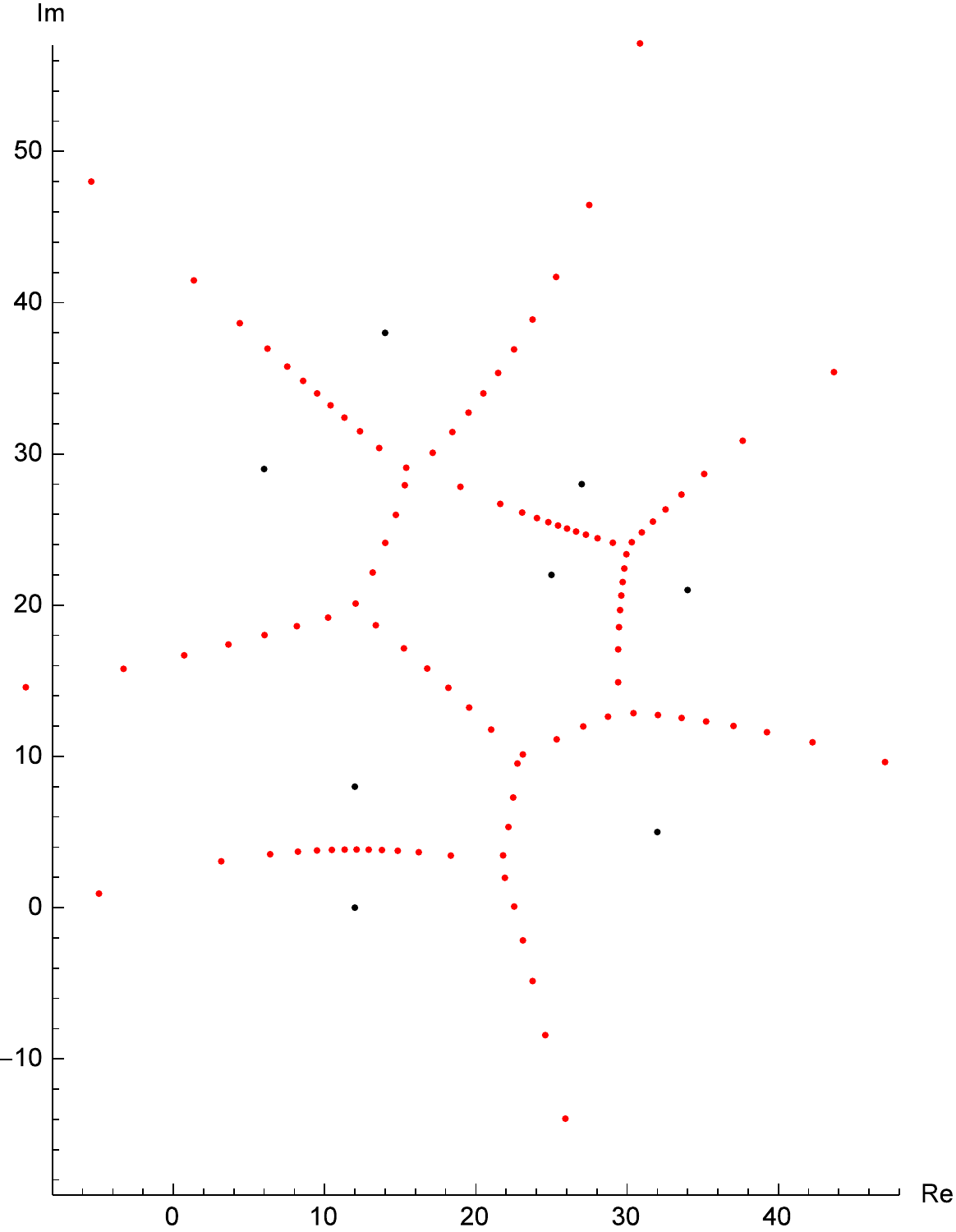} 
\end{center}
\caption{The Voronoi diagram of a polynomial $P$ of degree $8$ (left) and zeroes of $(1/P)^{(15)}$ (right).}
\label{fig1}
\end{figure}

An illustration of part (i) of this statement is given in Figure ~\ref{fig1}.

Statement (ii), for a sequence of measures implies (i), but not conversely. It follows from (ii) that the Cauchy transform of $\mu_n$ converges in $L^1_{loc}$ to the Cauchy transform of $\mu_{_S}$.

The sequence $Q_n$ may be thought of as a model (or toy) example of the asymptotic behaviour of zero-sets of sequences of polynomials or rational functions, in that it exhibits in an especially nice form features of more complicated examples. We emphasize the fact, contained in the above theorem, that the logarithmic potential of the asymptotic measure is the {\it global maximum} of a fixed finite number of harmonic functions, and not only a local maximum (as in e.g. \cite {BR, BoBo, BoSh}). In addition, (Section \ref{section: diffeq}), the Cauchy transform of the asymptotic measure trivially satisfies an algebraic equation (Corollary \ref{cor:2}).  This equation may be derived formally from a parameter dependent quasi exactly solvable differential equation of fixed degree, satisfied by the sequence of derivatives.

Perhaps the nicest feature of this approach to P\'olya's theorem through zero-counting measures, is the immediacy with which one obtains similar results for rational functions in arbitrary dimensions. Using the machinery of currents, we construct a sequence of multivariate polynomials, whose zero-sets converge to the Voronoi diagram associated with a generic hyperplane arrangement (Theorem \ref{theorem:currents}). This generalizes an example given by Demailly \cite{Dem82}. Our asymptotic measure is in certain cases a tropical current, as defined in the thesis of Babaee \cite{Babaee}.

In the last section, we have briefly discussed a similar application to polynomial lemniscates (see \cite{EbKh,PoRa}). We give a sequence of polynomials whose zero-counting measures converge to an explicit measure with support on a wide class of polynomial lemniscates. This generalization is outside the scope of P\'olya's theorem.

Finally we note that our work (see Example \ref{ex:counterexample}) provides a family of counterexamples to Conjecture 1 in  the recent paper by Shapiro and Solynin \cite{ShSo}.

The authors are indebted to Boris Shapiro, Dmitry Khavinson, Gleb Nenashev  and in particular Elizabeth Wulcan for discussions, ideas and comments, and to Alexander Eremenko and James Langley for answers to questions on P\'olya's theorem.

\section{Voronoi diagrams as the support of the Laplacian of superharmonic and subharmonic functions}
\subsection{Definition}
The Voronoi diagram (see \cite{Ok}), belonging to a finite set of points $S=\{ z_1,\dotsc,z_d\}\subset \bC$, is
a stratification of the complex plane, using the function
$$\Phi(z):=\smash{\displaystyle\min_{i=1,\dotsc,d}} \{|z-z_i|\}.$$
 The (closed) connected cell containing a $z_i\in S$, and so consisting of the points that are closest to $z_i$ of all the points in $S$, equals the set $$
 V_i=\{z: \Phi(z)=|z-z_i|\}.
 $$ Similarly, the boundary between two of these cells, $V_i$ and $V_j$, consists of a closed connected line segment that is a subset of the line
 $|z-z_i|=|z-z_j|$ :
 $$V_{ij}=\{z: \Phi(z)=|z-z_i|=|z-z_j|\}.
 $$

 Together these boundaries form the {\it 1-skeleton } $\Vor_{_S}$ of the diagram. When $S$ is the set of zeroes of a polynomial $P$ we will sometimes denote it by $\Vor_{_P}$. Finally, vertices of the Voronoi diagram are exactly the sets of points $z$ for which three or more distances $|z-z_i|,\ i=1,\dotsc,d$, coincide with $\Phi(z)$. An example is seen in Figure~1.

Clearly $\Phi(z)$ is a piecewise differentiable function with locus of non-differentiability consisting of $\Vor_{_S}$ and $S$.  
\subsection{The subharmonic function $\Psi$} 
\label{section:Psi}
 We will have essential use for a slightly different {\it subharmonic}  function $\Psi(z)$. Let 
$\tilde P=\prod_{i=1}^d(z-z_i)$ be the monic polynomial with simple zeroes at $z_1,\dotsc,z_d$.
Define the function
\begin{eqnarray*}
\Psi_S(z):=&\frac{1}{d-1}(-\log \Phi(z)+\log\vert \tilde P(z)\vert)\\
=&\frac{1}{d-1} \left(\smash{\displaystyle\max_{i=1,\dotsc,d}} \{\log( |z-z_i|^{-1})\}+\log\vert \tilde P(z)\vert\right).
\end{eqnarray*}
 Clearly $\Psi=\Psi_S$ has similar properties to $\Phi$.  
 \begin{lemma}
 $\Psi$ is a continuous subharmonic function, defined in the whole complex plane (even for points in $S$), and harmonic in the interior
of any cell $V_i$.
\end{lemma}
\begin{proof}
In fact, in the Voronoi cell $V_i$ we have
\begin{equation*}
(d-1)\Psi(z)=\log |z-z_i|^{-1}+\log\vert\tilde P(z)\vert,
\end{equation*}
hence $\Psi$ is continuous in $\bC$, and harmonic in the interior of $V_i$.
Also, $\Psi(z)$ is the maximum of a finite number of harmonic functions, and is thus subharmonic.
\end{proof} 

Since $\Psi$ is subharmonic, $\Delta\Psi=4\frac{\partial^2 \Psi}{\partial \bar z\partial z}$ is a positive measure, with support on the boundary of the cells, i.e. on $\Vor_{_S}$. 

\begin{proposition}
\label{lemma:1}
Define a measure with support on the line $l_{ij}:\ \vert z-z_i\vert =\vert z-z_j\vert$ as
$$
\delta_{ij}=\frac{1 }{4(d-1)}\frac{\vert  z_i-z_j \vert}{\vert (z-z_i)(z-z_j)\vert}\diff s.
$$
where $\mathrm{d}s$ is Euclidean length measure in the complex plane. Then
\begin{enumerate}
\item  $\frac{\partial^2 \Psi}{\partial \bar z\partial z}$ is the sum of all $\delta_{ij}$, each restricted to $V_{ij}$.

\item $\mu_{_S}:=\frac{2}{\pi}\frac{\partial^2 \Psi}{\partial \bar z\partial z}$ is a probability measure.
\end{enumerate}
\end{proposition}
We will later, in Corollary \ref{cor:probmeasure}, also see that $\Psi$ is the logarithmic potential of $\mu_{_S}$. For the proof of the proposition we need to recall some general facts on piecewise harmonic functions.

\subsection{The Laplacian of a piecewise harmonic function}
Assume that $\Psi$ is a continuous and subharmonic function, harmonic and equal to $H_i$ in the open sets $V_i,\ i=1,\dotsc,d$, that form a cover of $\bC$ a.e. Assume that each $H_i$ is defined and harmonic in an open set containing $\bar V_i$.
Let $A_i := \partial H_i/\partial z$. Assume finally that $\partial V_i=\cup_{j}V_{ij}$ is a finite union of $C^1$ curves, where $V_{ij}\subset \partial V_i\cap \partial V_j$ are closed as sets. We set $V_{ij}=V_{ji}$, and also define $V_{ij} = \emptyset$ if $V_i$ and $V_j$ have no common curve segment as a border between them. Orient $V_{ij}$ so that $V_{\min(i,j)}$
 is on the right of $V_{ij}$.

\begin{lemma} 
\label{lemma:piecewise}
 As distributions, acting on a test function $f\in C_c(\mathbb C)$,
 \begin{enumerate}
 \item $$\left\langle\frac{ \partial \Psi}{\partial z}, f\right\rangle =\sum_{i}\int_{V_{i}}A_i(z)f(z)\,\mathrm{d}x\wedge \mathrm{d}y.$$
 \item $$\left\langle\frac{ \partial^2 \Psi}{\partial \bar z\partial z}, f\right\rangle = \sum_{i<j}\frac{1}{2i}\int_{V_{ij}}(A_i(z)-A_j(z))f(z)\,\mathrm{d}z.$$
 
 \end{enumerate}
\end{lemma}
\begin{proof}
In an interior point $z\in V_i$, $\Psi(z)$ has a pointwise derivative  $A_i=\partial H_i/\partial z$, and hence $\Psi(z)$ has a pointwise derivative a.e. (with respect to Lebesgue measure).  This derivative is a $L^1_{loc}$-function, and so defines a distribution. For a general function the pointwise derivative does not necessarily coincide with the distributional derivative; however for a function that is continuous and piecewise differentiable, this is true (see e.g. Prop. 2 in \cite{BoBo}). This proves (1).

Hence
$$
\left\langle\frac{ \partial^2 \Psi}{\partial \bar z\partial z}, f\right\rangle = -\left\langle\frac{ \partial \Psi(z)}{\partial z}, \frac{ \partial f(z)}{\partial \bar z}\right\rangle = -\sum_{i}\int_{V_{i}}A_i(z)\frac{ \partial f(z)}{\partial \bar z}\,\mathrm{d}x\wedge \mathrm{d}y.
$$
But, since $A_i$ is holomorphic,
$$A_i(z)\frac{ \partial f(z)}{\partial \bar z}=\frac{ \partial (A_i(z)f(z))}{\partial \bar z}, \ z\in V_i,$$
so the last sum equals, using Stokes' theorem (\cite[1.2.2]{HoComplex}),
\begin{equation}
\label{eq:curveintegral}
-\frac{1}{2i}\sum_{i}\int_{\partial V_{i}}A_i(z)f(z)\,\mathrm{d}z.
\end{equation}
The curve integrals in this sum use an orientation of $\partial V_{i}$, such that $V_i$ is to the left of $\partial V_{i}$. Rewriting (\ref{eq:curveintegral}) as a sum of curve integrals on $V_{ij}, $ noting that each $V_{ij}$ occurs in both $\partial V_{i}$ and $\partial V_{j}$ with opposite orientation gives (2).
\end{proof}

\subsection{Proof of Proposition \ref{lemma:1}}
For $\Psi$ as in \ref{section:Psi},
\begin{equation}
\label{eq:Ana}
A_i(z):=\frac{\partial \Psi}{\partial z}=(2d-2)^{-1}\left(-(z-z_i)^{-1}+\sum_{j=1}^d(z-z_j)^{-1}\right),\ z\in V_i,
\end{equation}

and hence 
$$A_i(z)-A_j(z)=\frac{z_j-z_i}{(2d-2)(z-z_i)(z-z_j)}.$$

$V_{ij}$ is a segment of the orthogonal bisector of the line segment between $z_1$ and $z_2$, that is given by the equation
$z=(1/2)(z_i+z_j)-ti(z_j-z_i),\ t\in  \mathbb R$ (with the orientation that $V_i$ is to the right of $V_{ij}$). Assume that $z\in V_{ij}.$ Then 
$z-z_i=(z_j-z_i)(1/2-ti)$ as well as $z-z_j=(z_j-z_i)(-1/2-ti)$,  and since $\diff z=-i(z_j-z_i)\diff t$, 
\begin{equation}
\label{eq:part1}
(A_i(z)-A_j(z))\diff z=\frac{i\diff t}{(2d-2)(1/4+t^2)}.
\end{equation}

The Pythagorean theorem applied to the triangle with vertices $z,\,z_i,$ and $(1/2)(z_i+z_j)$ gives that $\vert z-z_i\vert^2=\vert z_j-z_i\vert^2((1/4)+t^2)$. Length measure along $V_{ij} $ is given by $\diff s=\vert z_j-z_i\vert \diff t.$ Inserting these expressions for $t$ and $\diff t$ in (\ref{eq:part1}), and applying Lemma \ref{lemma:piecewise} (2), gives part (1) of the proposition.

\begin{remark}{\rm One could also have used the Plemelj-Sokhotski relations, to directly derive (1), see \cite[Lemma 2]{BoBo} or (in a different form) Theorem II.1.5 (i) of \cite{SaffTotik}. }
\end{remark}

Now for the second part of the proof, 
let $C_R$ be the circle $\vert z\vert =R$, and $D_R$ be the disk $\vert z\vert \leq R$. Choose $R$ so large that all bounded cells in the Voronoi diagram are contained in the interior of $D_R$, and let $\chi_R$ be the characteristic function of $D_R$. 
By (\ref{eq:curveintegral}) in the proof of Lemma \ref{lemma:piecewise},
\begin{equation*}
\mu_{_S}(D_R)= \frac{2}{\pi}\left\langle\frac{ \partial^2 \Psi}{\partial \bar z\partial z}, \chi_R\right\rangle= -\frac{1}{\pi i}\sum_{i}\int_{\partial V_{i}\cap D_R}A_i(z)\,\mathrm{d}z.
\end{equation*}
The contribution of the boundaries of the bounded cells to the sum is zero, and the contribution of the unbounded cells $V_i,\ i=1,\dotsc,d$ equals the curve integral of $A_i(z) = \frac{\partial \Psi}{\partial z}$  on $V_i\cap C_R$ (both statements by Cauchy's theorem), so that
$$
\mu_{_S}(D_R)= \frac{1}{\pi i}\sum_{i=1}^d\int_{\partial V_{i}\cap C_R}A_i(z)\,\mathrm{d}z. 
$$
Note that the induced orientation on $C_R$ is positive. It remains only to exploit that $A_i(z)\sim 1/2z$ at infinity. Set $\frac{\partial \Psi}{\partial z}=(1/2z)+B(z)$. By (\ref{eq:Ana}),
$$
B(z)=\frac{1}{(2d-2)} \sum_{j\neq i}\frac{z_j}{z(z-z_j)},\quad\text{ for } z\in V_i,
$$
so that it, by trivial estimates, satisfies $\lim_{R\to\infty}\int_{C_R}B(z)\,\mathrm{d}z=0$. Hence $$
\lim_{R\to\infty} \mu_{_S}(D_R)=\frac{1}{2\pi i }\int_{C_R}\frac{\mathrm{d}z}{z}=1.
$$
\begin{remark}{\rm Since our goal is just to understand $\Psi$ we have not striven for generality in the previous proposition; it is clear that one could reformulate the above discussion as a general statement on piecewise harmonic and subharmonic functions, whose Laplacian has unbounded support. }
\end{remark}

\section{When the Voronoi diagram is a line}
 As an illuminating example of our theorem, we will  describe the easiest case, in which all calculations can be made explicit. Let
\begin{equation}
\label{eq:2poles}
Q(z):=\frac{A_1}{z-z_1}+\frac{A_2}{z-z}_2,\ A_1A_2\neq 0,\ z_1\neq z_2.
\end{equation}

 The Voronoi diagram associated with $\{ z_1,z_2\}$ consists of the line $V_{12},$ given by the condition $\vert z-z_1\vert=\vert z-z_2\vert$. This is the line that orthogonally bisects the line segment between $z_1$ and $z_2$. The M{\"o}bius transformation $$M(z):=\frac{z-z_1}{z-z_2}$$ maps $V_{12}$ injectively to the unit circle $C$,  with all $C$ except $1$ in the image. It will be crucial below. It has  inverse
$$
M^{-1}(z)=\frac{zz_2-z_1}{z-1}.
$$ 
 Normalized angular measure 
  $$\mathrm{d}\theta:=\frac{1}{2\pi}\,\mathrm{d}s,$$
(where $\mathrm{d}s$ is Euclidean length measure in the plane), seen as a measure on $\bC$ with support on $C$, pulls back by $z=M(w)$ to give a measure 
 $$
\mu:=\frac{1}{2\pi }\vert M'(w)\vert\,\mathrm{d}s=\frac{\vert z_1-z_2\vert}{2\pi\vert z-z_1\vert\vert z-z_2\vert}\,\mathrm{d}s
 $$
 with support on $V_{12}$. Recall that the pullback measure is defined by the property 
  \begin{equation}
\label{eq:2measures}
 \int f(M^{-1}(z))\,\mathrm{d}\theta=\int f(w)\,\mathrm{d}\mu,
\end{equation}
 where $f\in C_c(K)$ is a test function with support in the compact disc $K$.
 
 The pullback measure $\mu$ on $V_{12}$ turns out to be the asymptotic measure. The proof is quite explicit, since we can find the zeroes of
 $$
 Q^{(n-1)}(z)=\frac{(-1)^{n-1}(n-1)!( A_1(z-z_2)^{n}+ A_2(z-z_1)^{n})}{((z-z_1)(z-z_2))^{n}}, $$
  (for $ n\geq 1)$ as the solutions to
 $$
 A_1(z-z_2)^{n}+ A_2(z-z_1)^{n}=0 \iff \frac{z-z_1}{z-z_2}=(-A_1/A_2)^{\frac{1}{n}}\epsilon_n^{i}, \ i=0,1,\dotsc,n-1,
 $$
where $b_n:=(-A_1/A_2)^{\frac{1}{n}}$ is {\it one} choice of an $n$'th root, and $\epsilon_n$ is a primitive $n$'th root of unity.
Hence the zeroes are 
\begin{equation*}
\xi_i:=M^{-1}(b_n\epsilon_n^{i}), \ i=0,1,\dotsc,n-1.
\end{equation*}
(Note that possibly one $b_n\epsilon_n^{i}$ equals $1$,  so that $M^{-1}(b_n\epsilon_n^{i})$ is not defined. Then there are only $n-1$ zeroes of $Q^{(n-1)}$. This will mean an obvious, but inessential change in the argument below, which we will not detail.)

The following proposition is a special case of part (i) of the theorem.

\begin{proposition}
\label{lemma:zeroes2}
Assume that $Q$ is given by (\ref{eq:2poles}), and let $\mu_{n-1}=\frac{1}{n}\sum_{i=0}^{n-1}\delta_{\xi_i}$ be the zero-counting measure of $Q^{(n-1)}$as above. Then $\mu_{n-1}\to \mu$ as distributions (or measures).
\end{proposition}
\begin{proof} For $f\in C_c(K)$, $K$ a compact disc,
$$
\int_{\bC} f\,\mathrm{d}\mu_{n-1}=\frac{1}{n}\left(\sum_{i=0}^{n-1}f(M^{-1}(b_n\epsilon_n^{i}))\right) \to\frac{1}{2\pi} \int_0^{2\pi} f\left(M^{-1}\left(e^{it}\right)\right)\mathrm{d}t,
$$
where the last integral is just $ \int_{\mathbb C}f(M^{-1}(z))\,\mathrm{d}\theta$. Now the result follows from (\ref{eq:2measures}). 

\end{proof}

\begin{example} 
\label{example:realline}
Suppose that $S=\{\pm i\}$. Then $V_{12}$ is the real axis and $\mu_{_S}=\frac{1}{\pi}\frac{1}{1+t^2}\diff t$ where $\diff t$ is Lebesgue measure in $\mathbb R$. Clearly $\mu_{_S}$ is a probability measure. If furthermore $Q(z)=\frac{1}{1+z^2}$, the zeroes of $Q^{(n)}(z)$ are $z=\cot\frac{k\pi}{n+1},\ k=1,\dotsc,n$ (see P\'olya \cite{Po1}).
In the upper half-plane $\Psi(z)=\log\vert z+i\vert $,
and in the lower half-plane $\Psi(z)=\log\vert z-i\vert $.
\end{example}

\section{Proof of the theorem}
As we have noted earlier, it is enough to prove the $L^1_{loc}$-convergence $L_n\to \Psi $ in Theorem 1.1 (ii) on a compact set $K$. We give an overview of the proof.
First we describe the derivatives $Q^{(n)}$ of the rational function $Q$, in section $4.1$, in particular the role that the largest pole plays (Lemma 4.1).  We also need to estimate the degree of the nominator $R_n$ of $Q^{(n)}$ and the highest coefficient $\alpha_n$ (Lemma 4.2-3). This is trivial in the generic case. Then we easily get uniform convergence on compact subsets of open Voronoi cells in Lemma 4.4 and Proposition 4.5.  This proves convergence a.e. In order to also prove $L^1_{loc}$-convergence it essentially only remains to note that the zeroes of $R_n$, where $L_n$ has singularities, do not grow too quickly to infinity (Lemma 4.6), and once this is done, to estimate the contribution of each of these zeroes to the integral of $L_n$ on a small neighbourhood of the Voronoi diagram. This is done in 4.4.2. Finally we note that $\Psi$ is actually the logarithmic potential of $\mu$, and describe the case of a single pole.

\subsection{Derivatives of a rational function}
A rational function with poles $z_1,\dotsc, z_d$, may be decomposed into polar parts as
\begin{equation*}
Q= P_0(z)+\frac{P_1(z)}{(z-z_1)^{r_1}}+\dotsb+\frac{P_d(z)}{(z-z_d)^{r_d}},
\end{equation*}
where $P_i\in \bC [z]$ has degree $s_i$ strictly less than $r_i,\ i=1,\dotsc,d$. Clearly $P_0(z)$ will not play any role for the asymptotic properties of $Q^{(n)}$, so we may assume $P_0(z)=0$. Write
$P_i(z)=\sum_{j=0}^{s_i}a_{i,r_i-j}(z-z_i)^{j}$ (with $a_{i,r_i}\neq 0$), so that
$$Q_i:=\frac{P_i(z)}{(z-z_i)^{r_i}}=\frac{a_{i,r_i-s_i}}{(z-z_i)^{r_i-s_i}}+\dotsb+\frac{a_{i,r}}{(z-z_i)^{r_i}}.$$
Hence
\begin{equation}
\label{eq:polarpart1}
{Q_i^{(n)}}=\sum_{j=1}^{r_i}\frac{a_{i,j}(-1)^n(j)_n}{(z-z_i)^{j+n}},
\end{equation}
where $(j)_n=j(j+1)\cdots(j+n-1)$  is the (ascending) Pochhammer symbol (as used in the theory of special functions).
Below we will need the following property:
\begin{equation}
\label{eq:factorial}
(j)_n=j(j+1)\cdots(j+n-1)=n!\frac{(n+1)\cdots(n+j-1)}{(j-1)!} =:n!p_j(n),
\end{equation}
($j\geq 1$) where $p_j(n)$ is a polynomial of degree $j-1$ in n.

The most singular term of the polar part will dominate the rest of $Q_i^{(n)},$ in the following way.
\begin{lemma}
\label{lemma:polarpart}
\begin{equation}
\label{eq:polarpart}
{Q_i^{(n)}}=\frac{(-1)^n(r_i)_na_{i,r_i}}{(z-z_i)^{r_i+n}}(1+S_i^n(z)),
\end{equation}
where the polynomial $S_i^n(z)\to 0$, uniformly on compact sets $K\subset \bC$, as $n\to \infty.$
\end{lemma}
\begin{proof}
It follows from (\ref{eq:polarpart1}) that the coefficients of the polynomial 
$$S_i^n(z)=\sum_{j=1}^{r_i-1}\frac{a_{i,j}(j)_n}{a_{i,r_i}(r_i)_n}(z-z_i)^{r_i-j}$$ tend uniformly to $0$ with increasing $n$, since
$$\lim_{n\to\infty}\frac{(j)_n}{(r_i)_n}= \lim_{n\to\infty} \frac{r_i!}{j!}\cdot   \frac{(n+j-1)!}{(n+r_i-1)!}=0,$$
if $j<r_i$. This implies the lemma.
\end{proof}

\subsection{Number of zeroes}
Let $Q$ be as above. Its derivative equals 
\begin{equation}
\label{eq:ratfuncdeg}
{Q^{(n)}}=\frac{\alpha_nR_n}{PP_0^n},
\end{equation}
where $P=\prod_{i=1}^d(z-z_i)^{r_i}$, $P_0=\prod_{i=1}^d(z-z_i)$, $\alpha_n\in\bC$ and $R_n\in \bC[z]$ is a monic polynomial. Let $r:=\deg P=r_1+\dotsb+r_d$. Note that $GCD(R_n,PP_0^n)=1$, by the fact that the $Q_i^{(n)},\ i=1,\dotsc,d$ have relatively prime denominators.
For generic $Q$ the degree of $R_n$ is $n(d-1)+r-1$. In general the asymptotic result below holds. 

\begin{lemma}
\label{lemma:degrn}
\begin{equation}
\label{eq:degrn}
\lim_{n\to\infty}\frac{\deg R_n}{n}=d-1.
\end{equation}
\end{lemma}
\begin{proof}
Clearly, from (\ref{eq:polarpart1}) 
\begin{equation}
\label{eq:firsteq}
\max_{i=1,\dotsc,d} \left\{ \sum_{j\neq i} (n+r_j)+r_i-1\right\}=(d-1)n+r-1\geq
 \deg R_n,
\end{equation}
and with generic coefficients of the rational function this will be an equality. Our proof of the general case proceeds in two steps.

{\bf Step 1.} Change variable to $y=1/z$. Then using (\ref{eq:polarpart1}), the fact that $r\geq r_i$, and (\ref{eq:factorial}) we can rewrite
\begin{equation*}
Q^{(n)}(1/y)=\sum_{i=1}^d\frac{S_i(y,n)y^n}{(1-z_iy)^{n+r}}:=n!y^n S(y,n)
\end{equation*}
where $S_i(y,n), \ i=1,\dotsc,d$ are polynomials in the variables $y$ and $n$. For each $n\in \mathbb Z$, $S(y,n)$ is a rational function in $y$, and it suffices to find a uniform bound (in $n$) of the order $o(S(y,n))$ of the zero of $S(y,n)$ at $y=0$.
For if $o(S(y,n)) \leq C$, then by (\ref{eq:ratfuncdeg})
\begin{eqnarray*}
&&\deg_{z} P_0P^n-\deg_{z} R_n=o(S(y,n)) +n 
 \implies \\
&&\deg_{z} R_n=r+nd-n-o(S(y,n)) \geq (d-1)n+ (r-C).
\end{eqnarray*}
This together with the upper bound (\ref{eq:firsteq}), implies the lemma.

{\bf Step 2.} Note that the highest non-zero coefficient of $R_n(z)$ is the lowest non-zero coefficient of $Q_n(1/y)$,  expanded as a power series in $y$, and, in addition, equals $n!b_s$, where $S(y,n)=\sum_{i=s}^\infty b_i(n)y^i$ and $ b_s(n)\neq 0$. 
The coefficients $b_i(n)$ are easily checked to be polynomials in $n$---this follows from Newton's binomial theorem, applied to the denominators---and thus have only a finite number of zeroes. Hence there exists an $N$ such that
$b_s(n)$ is non-zero for $n\geq N$, and consequently $o(S(y,n))=s$ if $n\geq N$. By the previous step we are done.
\end{proof}
The following consequence of the proof will be needed in the next section.
\begin{lemma} As  $n\to\infty$,
$$\frac{\log \vert n! /\alpha_n\vert }{n}\to 0.$$
\label{lemma:highest coefficient}
\end{lemma}

\begin{proof} As noted in the proof of the preceding lemma, the highest non-zero coefficient $\alpha_n$ of $R_n(z)$ is the lowest non-zero coefficient $n!b_s$ of the power series of $S(y,n)$. Write $b_s(n)=\sum_{i=0}^ka_in^i=a_kn^{k}\tilde b_s(n)$, with $a_k\neq 0$, and note that $\tilde b_s(n)\to 1$ as $n\to \infty$.
This implies the lemma.
\end{proof}

\subsection{Uniform convergence almost everywhere of the logarithmic potential} P\'olya \cite[Behauptung b),  p. 41]{Po1} gives
the following lemma (which we only prove for rational functions). We will use $V_i^o$ to denote an open Voronoi cell $V_i$.
 \begin{lemma}
  \label{lemma:polya} Let $Q(z)$ be a meromorphic function. Then the sequence $\left\vert\frac{Q^{(n)}(z)}{n!}\right\vert^{\frac{1}{n}}$ converges to
  $\vert z-z_i\vert^{-1}$ pointwise in $V_i^o\setminus\{z_i\}$. The convergence is uniform on compact subsets $W\subset V_i^o\setminus\{z_i\}$.
\end{lemma}
\begin{proof} For $Q$ a rational function, this is a simple consequence of Lemma \ref{lemma:polarpart}. By the definition of an open cell $V_i^o$, there is a $C<1$ such that
$$
z\in V_i^o\implies \frac{\left\vert (z-z_j)^{-1}\right\vert }{\left\vert (z-z_i)^{-1}\right\vert }<C, \ j\neq i,$$ 
and furthermore there is for any fixed compact $W\subset V_i^o\setminus\{z_i\}$ a $C<1$ that works for any $z\in W$. Since furthermore $(r_j)_n/(1)_n$ is a non-zero polynomial of fixed degree in $n$, it is clear that $((r_j)_n/(1)_n)C^n\to 0$ when $n\to\infty$. Hence by Lemma \ref{lemma:polarpart}, if $j\neq i$, and $z\in W$,
$$
\left\vert\frac{Q_j^{(n)}(z)}{(1)_n(z-z_i)^{-(r_i+n)}}\right\vert\leq M\frac{(r_j)_n}{(1)_n}C^{n+r_i}\to 0,
$$
when $n\to \infty$, where $M$ is the maximal value of $\vert a_{j,r_j}(z-z_j)^{r_i-r_j}\vert $ in $W$. Since $(1)_n=n!$,  another application of Lemma \ref{lemma:polarpart}, this time to $Q_i(z)$, gives that $$
\lim_{n\to\infty}\left\vert \frac{Q_i^{(n)}(z)}{n!(z-z_i)^{-(r_i+n)}}\right\vert^{1/n}=1$$
and this convergence is clearly also uniform in $W$. Since $\lim_{n\to\infty}\left\vert (z-z_i)^{-(r_i+n)}\right\vert^{1/n}=\vert z-z_i\vert^{-1}$, this implies the lemma.
\end{proof}
In particular there is only a finite number of zeroes of any $Q^{(n)}$ in a compact subset of an open cell of the Voronoi diagram.

 \begin{proposition}
  \label{lemma:conv.a.e} Let $L_n(z)=\log\vert R_n(z)\vert/\deg R_n$ be the logarithmic potential of the probability measure  $\mu_n$ associated with $R_n$. Then for any $z$ in the interior of the Voronoi cell $V_i^o$ we have pointwise convergence
  $$\lim_{n\to \infty}L_n(z) = (d-1)^{-1}\left(\log |P_0| + \smash{\displaystyle\max_{i=1,\dotsc,d}} \{\log |z-z_i|^{-1}\}\right)=:\Psi(z).$$
  The convergence is uniform on compact subsets of $V_i^o$.
\end{proposition}
\begin{proof} By Lemma \ref{lemma:degrn} it suffices to show that 
$\log\vert R_n(z)\vert/n\to \log |P_0| + \smash{\displaystyle\max_{i=1,\dotsc,d}} \{\log |z-z_i|^{-1}\}=(d-1)\Psi(z)$. Given a compact subset
$W\subset V_i^o$, there is by the preceding Lemma an $N_W$ such that for $n>N_W$ , no zeroes of $R_n$ are contained in $W$. Hence for $n>N_W$, both $\log\vert R_n(z)\vert/n$ and $(d-1)\Psi(z)$ are continuous and harmonic functions in $W$.
If $z\in W\subset V_i^o\setminus\{z_i\}$, (\ref{eq:ratfuncdeg}) implies that
$$
S_n(z):=\log\vert R_n(z)\vert/n= \log\left\vert \frac{Q_n(z)}{n!}\right\vert/n+ \log\left\vert \frac{n!}{\alpha_n}\right\vert/n +\log\vert P_0(z)\vert+ \log\vert P(z)\vert/n.
$$

By taking logarithms in the preceding lemma, the first term in this sum converges uniformly in $W$ to $\log (\vert z-z_i\vert^{-1})$. By Lemma \ref{lemma:highest coefficient}, the second term converges uniformly to $0$. The continuity of $P(z)$ on $W$ implies that this is also true of the last term. Hence the left-hand side converges uniformly to $(d-1)\Psi(z)$ in $W$. This can be extended to compact subsets of $V_i^o$. Assume that $\tilde W\subset V_i^o$ is  a compact set. Furthermore, assume that $z_i$ is contained in its interior,  and that the open disk $D(z_i, r)\subset \tilde W$. Let $W=\tilde W\setminus D(z_i, r)$, and use the previous result to get $N$ such that $n\geq N$ implies $-\epsilon <  S_n(z)-(d-1)\Psi(z)<\epsilon$
in $W$. Since $W$ contains the boundary of $D(z_i, r)$ and both $S_n(z)$ and $(d-1)\Psi(z)$ are harmonic in $D(z_i, r)$, it follows by the maximum principle that $-\epsilon <  S_n(z)-(d-1)\Psi(z)<\epsilon$ in $\tilde W$. This proves the proposition.
\end{proof}

\subsection{Proof of the main theorem}
\label{section:proof}

Uniform convergence a.e. as in Proposition \ref{lemma:conv.a.e}  does not by itself imply convergence of the logarithmic potentials in $L^1_{loc}$, though it tells us that there is only one possible limit, since a function in $L^1_{loc}$ is determined by its behavior a.e.  If we had a uniform bound on the $L^1$-norms of $\Psi_n$ on a compact set $K$ we could use sequential compactness, but it turns out to be easiest to give a direct proof of the convergence. The main obstacle is that the zeroes of $R_n$ are unbounded as $n\to \infty$. To handle this we give a rough bound of the growth of the zeroes of $R_n$ in Lemma \ref{lemma:growth}.
\subsubsection{Growth of zeroes}
\label{section:growth}
Fix a rational function $R(z)$. Note that if we have proven the statement of the theorem for $R(z)$, then the theorem also follows for  $S(w)=R(z)$, where we have changed variables by setting $w=z+a$. There will be at least one open component of the complement $\Vor_{_S}^c$ of the Voronoi diagram, with the property that it contains open disks $D_R(a)$ with center $a $ and radius $R$ of arbitrarily large radius. Hence we may assume, by choosing $a$ suitably, and so moving the Voronoi diagram, that the following holds:

\medskip
(*) the closed unit disk $\bar D_1=D_1(0)\subset \Vor_{_S}^c$. Consequently, by Lemma \ref{lemma:conv.a.e}, there is a number $N$ such that  $z\in D_1\implies R_n(z)\neq 0$, if $n\geq N$.  Equivalently, if $n\geq N$ and $R_n(z)=0$, then $\vert z\vert \geq 1$.

 \medskip
(**) $\bar D_1$ contains no pole, hence $\vert z_i\vert \geq 1,\,i=1,\dotsc,d$.

\medskip

For $K\subset \bC$, set 
$$\vert z_{K,n}\vert =\prod_{z\in K: R_n(z)=0}\vert z\vert, $$
(zeroes taken with multiplicities, and if there are no zeroes of $R_n(z)$ in $K$, then $\vert z_{K,n}\vert =1$. Let $D_R=\{ z: \vert z\vert \leq R\}$, for $R>0$, and let $m_n:=\deg R_n$.

\begin{lemma}
Assume (*) and (**). For $K$ equal to $D_R$ or $D_R^c$,  $\vert z_{K,n}\vert ^{1/n}$ is a bounded sequence.
Hence there is  a number $C$ such that $0\leq (1/m_n)\log \vert z_{K,n}\vert \leq C$ if $n\geq N$.
\label{lemma:growth}
\end{lemma} 
\begin{proof} By (*), $1\leq \vert z_{K,n}\vert ^{1/n}$, for $K$ equal to either $D_R$ or $D_R^c$ and $n\geq N$.  Furthermore
$\vert z_{D_R,n}\vert \vert z_{D_R^c,n}\vert =\vert R_n(0)\vert $, so it suffices to check that $\vert R_n(0)\vert^{1/n}$ is a bounded sequence. We have, by (\ref{eq:ratfuncdeg}) and Lemma \ref{lemma:polarpart}, 

\begin{equation*}
R_n(0)=\left(\frac{n!}{\alpha_n}\right)P_0(0)P(0)^n \sum_{i=1}^d\frac{(-1)^n\frac{(r_i)_n}{n!}a_{i,r_i}}{(-z_i)^{r_i+n}}\big(1+S_i^n(0)\big).
\end{equation*}
Consider the factors.
First, by Lemma \ref{lemma:highest coefficient}, $\vert \frac{n!}{\alpha_n}\vert^{1/n}\to 1$ as $n\to\infty$. Second, $\vert P_0(0)P(0)^n\vert^{1/n}\to \vert P(0)\vert \neq 0$ by (*). Finally, by Lemma \ref{lemma:polarpart},  $\vert 1+S_i^n(0)\vert$ is bounded, and so there is an $M> 0$ such that
$$
\left\vert\sum_{i=1}^d\frac{(-1)^n\frac{(r_i)_n}{n!}a_{i,r_i}}{(-z_i)^{r_i+n}}\big(1+S_i^n(0)\big)\right\vert\leq M \max_{i=1,\dotsc,d}\left\{\frac{(r_i)_n}{z_i^nn!}\right\}.
$$
Since $\frac{(r_i)_n}{n!}$ is a polynomial in $n$ of fixed degree, it follows that $\vert R_n(0)\vert^{1/n}$ is indeed bounded. This completes the proof.

\end{proof}
\subsubsection{A lot of integrals}
\label{sec:3integrals}
We now turn to the proof of the theorem, and first note that it is enough to prove (ii), since (i) is an immediate consequence by taking the Laplacian.
We want to estimate, for a fixed $R>0$,
\begin{equation*}
I_1:=\int_{\D_R}\vert L_n(z)-\Psi(z)\vert\,\mathrm{d}\lambda .
\end{equation*}

Fix $0<\epsilon<1$. By the uniform convergence in Proposition \ref{lemma:conv.a.e}, there is to any open subset $U\subset \bar U \subset D_R\setminus \Vor_{_S}$ an $N$ such that $n\geq N$ implies that
$\vert L_n(z)-\Psi(z)\vert\leq \epsilon$ if $z\in U$. Hence\begin{equation}
\label{eq: integralone}
I_2:=\int_{U}\vert L_n(z)-\Psi(z)\vert\,\mathrm{d}\lambda \leq \pi R^2\epsilon=O(\epsilon).
\end{equation}

Furthermore 
\begin{equation}
\label{eq:twointegrals}
\int_{D_R\setminus U}\vert L_n(z)-\Psi(z)\vert\,\mathrm{d}\lambda \leq \int_{D_R\setminus U}\vert L_n(z)\vert\,\mathrm{d}\lambda +\int_{D_R\setminus U}\vert \Psi(z)\vert\,\mathrm{d}\lambda =: I_4+I_3
\end{equation}

The last integral \begin{equation}
\label{eq:secondintegral}
I_3\leq M_1\lambda(D_R\setminus U)\leq M_1 \epsilon,
\end{equation}
 letting $M_1$ be the maximal value of $\Psi(z)$ in $D_R$, and assuming that $U$ is chosen suitably large, so that 
 \begin{equation}
 \label{eq:epsilon}
\lambda(D_R\setminus U)\leq \epsilon.
\end{equation}
The integral $I_4$ in (\ref{eq:twointegrals}) needs more bookkeeping. First split the function into two parts
 $$
 L_n(z)=\sum_{\vert z_i\vert \geq R+1}(1/m_n) \log\vert z-z_i\vert+\sum_{\vert z_i\vert <R+1}(1/m_n) \log\vert z-z_i\vert:=L_n^{I}(z)+ L_n^{II}(z).
 $$
 Then from
 \begin{equation*}
0\leq \log \vert z-z_i\vert\leq \log( R+\vert z_i\vert)\leq \log R+\log\vert z_i\vert,\hskip 0.5 cm \text{if } \vert z_i\vert\geq R+1,\ \vert z\vert\leq R .
\end{equation*}
we get that 
\begin{eqnarray}
 &&\int_{D_R\setminus U}\vert L_n^{I}(z)\vert\,\mathrm{d}\lambda = \sum_{\vert z_i\vert \geq R+1}(1/m_n)\int_{D_R\setminus U} \vert \log\vert z-z_i\vert \vert\,\mathrm{d}\lambda\\
&\leq& \left(\log(R+1)+  \sum_{\vert z_i\vert \geq R+1}\log \vert z_i\vert (1/m_n)\right) \lambda(D_R\setminus U)\\
  \label{eq:thirdintegral1}
 &\leq &(\log(R+1)+C)\epsilon=O(\epsilon),
 \end{eqnarray}
where the last inequality follows from Lemma \ref{lemma:growth}, using (\ref{eq:epsilon}).

If in addition to $\vert z\vert\leq R$ and $\vert z_i\vert\leq R+1$, also $\vert z-z_i\vert \geq \epsilon$,
 then $\vert \log\vert z-z_i\vert \vert\leq \max\{ -\log\epsilon, \log(2R+1)\}$.
This implies the inequality
\begin{eqnarray*}
&&\int_{D_R\setminus U}\vert \log\vert z-z_i\vert \vert\,\mathrm{d}\lambda\\
&&\leq \int_{\vert z-z_i\vert
\leq \epsilon}\vert \log\vert z-z_i\vert \vert\,\mathrm{d}\lambda+
\max\{ -\log\epsilon, \log(2R+1) \}\lambda(D_R\setminus U)\\
&&
\leq-2\pi(\log\epsilon-1/2)(\epsilon^2/2)+\max\{ -\log\epsilon, \log(2R+1) \}\epsilon=o(1).
\end{eqnarray*}
Note that the bound does not depend on $z_i$. Hence if we sum over all terms in the integral of $L^{II}$, which are at most $m_n$ in number, we get that
\begin{equation}
\label{eq:thirdintegral2}
\int_{D_R\setminus U}\vert L_n^{II}(z)\vert\,\mathrm{d}\lambda=o(1).
\end{equation}

Now Theorem 1.1(ii) follows from the fact that the upper bounds in (\ref{eq: integralone}), 
(\ref{eq:secondintegral}), (\ref{eq:thirdintegral1})  and (\ref{eq:thirdintegral2}) go to 0 when $\epsilon$ goes to 0.

 \subsection{$\Psi$ is a logarithmic potential}
 Let $\mu_{_S}$ be the measure defined earlier on the Voronoi diagram associated with a set of points $S$.
Given that $L:=L_{\mu_{_S}}(z)=\int_{\bC} \log\vert z-\zeta\vert \diff \mu_{_S}(\zeta)$ exists as a $L^1_{loc}-$function,  it will differ from $\Psi_S$ by a harmonic function. But in fact they are equal.
 \begin{cor} 
 \label{cor:probmeasure}$\Psi_{_S}$ is the logarithmic potential of $\mu_{_S}$.
\end{cor}
\begin{proof}
First of all, $L:=L_{\mu_{_S}}(z)=\int_{\bC} \log\vert z-\zeta\vert \diff \mu_{_S}(\zeta)$ is well-defined as a $L^1_{loc}-$function: Let $l_{ij}= \{ z:\ \vert z-z_i\vert =\vert z-z_j\vert\}$, and use the notation of Proposition \ref{lemma:1}. Then, for $K\subset \bC$ a compact set,
$$
\int_K\vert L(z)\vert d\lambda(z) \leq\sum_{i,j} \int_{l_{ij}} \left(\int_K\vert \log\vert z-\zeta\vert \vert\diff \lambda(z)\right)\diff\delta_{ij}(\zeta). 
$$ 
Now fix a line $l_{ij}$. An affine change of coordinates transforms $l_{ij}$ into the real axis, and then $\delta_{ij}$ is given by
$\frac{1 }{\pi}\frac{1}{1+t^2}\diff t$ (see Example \ref{example:realline}). Hence it suffices to prove that
$$
\int_{\bR} \left(\int_K\frac{\vert \log\vert z-t\vert \vert}{1+t^2}\diff \lambda(z)\right)\diff t
$$
is finite. This is clear, since for large $\vert t\vert $, the integrand is approximately $\lambda(K)\log\vert t\vert /t^2$.
Secondly, we will prove that $L(z)$ has the property that 
$$\lim_{\vert z\vert \to\infty}(L(z)-\log\vert z\vert)=0.$$
Since $\Psi(z)$, by inspection, has the same property, the desired result  directly follows: the harmonic function $\Psi(z)-L(z)$ is bounded, and hence constant and equal to $0$.

Now, as above,
$$
\vert L(z)-\log\vert z\vert\vert \leq\sum_{i,j} \int_{l_{ij}}{\left\vert \log\left\vert 1-\frac{\zeta}{z}\right\vert \right\vert}\diff\delta_{ij}(\zeta),
$$
and again, after an affine transformation, it is enough to consider
$$
\int_{\bR}\frac{\left\vert \log\left\vert 1-\frac{t}{z}\right\vert \right\vert}{1+t^2}\diff t,
$$
which is easily seen to have the limit $0$ as $\vert z\vert\to\infty.$

\end{proof}

\subsection{A single pole}
For completeness, we consider the case when $Q$ has only one pole. 

\begin{lemma}
\label{lemma:multipole}
Let $R(z)$ be a polynomial with $\tilde{d}\ge 0$ zeroes (distinct or multiple), let $m$ and $n$ be nonnegative integers, let $\alpha\in\mathbb{C}$, and define
\begin{equation*}\label{eq:explicitpoleNoDerivs}
Q(z) = \frac{R(z)}{(z-\alpha)^m}.
\end{equation*}
Then for any $r\in\mathbb{R}_+$, the open disk $\vert z\vert < r$ contains no zeroes of $Q^{(n)}(z)$ for all sufficiently large $n$.
\end{lemma}
\begin{proof}
Without loss of generality, we can assume that $\alpha$ is not a zero of $R(z)$ and that $\tilde{d}\ge 1$. By using the Taylor expansion of $R(z)$ about $z=\alpha$ and summing the terms where $k-m<0$ and $k-m\ge 0$ separately, we see that
\begin{equation}\label{eq:explicitpole}
\begin{split}
Q^{(n)}(z) = \left(\frac{R}{(z-\alpha)^m}\right)^{(n)} & = \frac{(-1)^n}{(z-\alpha)^{m+n}}\sum_{k=0}^{m-1}\frac{R^{(k)}(\alpha)}{k!}\cdot\frac{(n+m-k-1)!}{(m-k-1)!}\,(z-\alpha)^k \\
& + \sum_{k=m}^{\tilde{d}}\frac{R^{(k)}(\alpha)}{k!}\left((z-\alpha)^{k-m}\right)^{(n)}.
\end{split}
\end{equation}
Note from (\ref{eq:explicitpole}) that $Q^{(n)}(z)$ has $\theta := \min(m-1,\tilde{d})$ zeroes for all large enough $n$, since the polynomial $H_n(z) := \sum_{k=0}^{\theta} L_{n,k}(z-\alpha)^k$ dominates its zero distribution for such $n$, where
\begin{equation*}
L_{n,k} := \frac{R^{(k)}(\alpha)}{k!}\cdot\frac{(n+m-k-1)!}{(m-k-1)!}.
\end{equation*}
In particular, if $m=1$, it follows from (\ref{eq:explicitpole}) that $Q^{(\tilde{d})}(z) = -R(\alpha)\tilde{d}!/(\alpha-z)^{\tilde{d}+1}$. Thus, any rational function with a simple pole has no zeroes when differentiated $n\ge \tilde{d}$ times.

If $m>1$, consider the polynomial
\begin{equation*}
W_n(y) := y^{\theta}H_n(1/y+\alpha) = \sum_{k=0}^{\theta}L_{n,k}y^{\theta - k} = \sum_{k=0}^{\theta}L_{n,\theta-k}y^{k}.
\end{equation*}
Clearly, when $n\to\infty$, the zeroes of $Q^{(n)}(z)$ and $H_n(z)$ diverge toward $\infty$ iff the zeroes of $W_n(y)$ converge on $0$. Now note that, for $j=0,1,\dotsc,\theta-1$,
\begin{equation}\label{eq:fujiwara}
\lim_{n\to\infty}\left\vert\frac{L_{n,\theta-j}}{L_{n,0}}\right\vert^{\frac{1}{\theta-j}} = \lim_{n\to\infty}\left\vert\frac{C}{(n+m-1)(n+m-2)\dotsm (n+m-\theta+j)}\right\vert^{\frac{1}{\theta-j}} = 0,
\end{equation}
where
\begin{equation*}\label{eq:fujiwaraLongEq}
C := \frac{R^{(\theta-j)}(\alpha)}{R(\alpha)}\cdot\frac{(m-1)!}{(\theta-j)!\,(m+j-\theta-1)!}
\end{equation*}
is a constant independent of $n$. Thus, equation (\ref{eq:fujiwara}) and Fujiwara's bound \cite{Fu} show that all zeroes of $W_n(y)$ converge on $0$ when $n\to\infty$.
\end{proof}

\section{Algebraic and differential equations}

\subsection{The algebraic equation for the Cauchy transform of the asymptotic measure}
\label{section: diffeq}

The Cauchy transform $C_\mu$ of the measure $\mu$ is defined as  
\begin{equation*}
C_\mu(z)=\int \frac{1}{z-t}\,\mathrm{d}\mu(t),
\end{equation*}
and satisfies $C_\mu=2\frac{\partial \Psi}{\partial z}$ and
$\mu=\frac{1}{\pi}\frac{\partial C_\mu}{\partial \bar z}$, where $\Psi$ is the logarithmic potential of $\mu$.

\begin{corollary}
\label{cor:2}
 Assume $d\geq 2$. The Cauchy transform of the asymptotic measure $\mu=\mu_{_S}$ associated with the set $S=\{z_1,\dotsc,z_d\}$ satisfies the
algebraic equation
\begin{equation*}
\prod_{i=1}^d\left(C_\mu-(d-1)^{-1}\left(\sum_{j\neq i}^d(z-z_j)^{-1}\right)\right)=0
\end{equation*}
a.e. 
\end{corollary}
\begin{proof} It is part of the proof of Proposition \ref{lemma:1}  (and the theorem) that the Cauchy transform $C_\mu=2\frac{\partial \Psi}{\partial z}$  is piecewise analytic. In the interior of the Voronoi cell $V_i$
it equals $$
(d-1)^{-1}\left(-(z-z_i)^{-1}+\sum_{j=1}^d(z-z_j)^{-1}\right)=(d-1)^{-1}\left(\sum_{j\neq i}^d(z-z_j)^{-1}\right).
$$
Since such interiors cover $\bC$ a.e, the equation follows.

\end{proof}
An algebraic equation for the Cauchy transform of an asymptotic measure is an interesting invariant, from which the local behaviour of the logarithmic potential may be deduced, see e.g  \cite{BoSh}. Here the solutions of the algebraic equation have no monodromy, since they are rational functions, making the situation very simple. Algebraic equations may sometimes be derived from differential equations of the Cauchy transform; such exist also here, as we will see now.

\subsection{Differential equations for $Q^{(n)}$ and $R_n$}
\label{section:diffeq}
We will only give explicit equations in a special case, and in the general case we just indicate how they can be produced.
Let
$$
Q=\sum_{j=1}^d\alpha_j(z-z_i)^{-s}.
$$
Then
\begin{equation}
\label{eq:qder}
Q^{(n)}=(-1)^n(s)_n\sum_{j=1}^d\alpha_j(z-z_i)^{-s-n}.
\end{equation}
These are power sums, and Newton's relations imply that there is a differential equation satisfied by $Q^{(n)}$.

\begin{proposition} 
\label{prop:diffeq}If the rational function $Q=\sum_{j=1}^d\alpha_j(z-z_i)^{-s}$, then its derivatives $Q^{(n)}$ satisfy
\begin{equation}
\label{eq:qderspecial}
\sum_{i=0}^d \frac{e_i(z) }{(s+n)(s+n+1)\dotsm(s+n+i-1)}\frac{\diff^iQ^{(n)}}{\diff z^i}=0,
\end{equation}
where $e_i(z)$ is the $i$'th elementary symmetric function in the $d$ arguments $z-z_i,\quad i=1,\dotsc,d$ (and $e_0=1$).
\end{proposition}
\begin{proof} 
Let $t=w_j$ in $\prod_{j=1}^d(t-w_j)=\sum_{i=0}^d (-1)^{i}e_i(w) t^{d-i}$, which gives
$$
\sum_{i=0}^d (-1)^{i}e_i(w) w_j^{d-i}=0.
$$
Divide by $w_j^{d+s+n}$ and  modify the coefficients:
\begin{equation}
\label{eq:newton}
\sum_{i=0}^d (-1)^{i}\frac{(-1)^{n+i}e_i(w) }{(s)_{n+i}} (-1)^{n+i}{(s)_{n+i}}w_j^{-s-n-i}=0.
\end{equation}

Hence, letting $w_j=z-z_j$, multiplying by $\alpha_j$ and adding the equations (\ref{eq:newton}) for $j=1,\dotsc,d$, shows, using (\ref{eq:qder}), that $y=Q^{(n)}$ satisfies
$$
(-1)^n\sum_{i=0}^d \frac{e_i(z) }{(s)_{n+i}}\frac{\diff^iy}{\diff z^i}=0.
$$
This is equivalent to $(\ref{eq:qderspecial})$.
\end{proof}

It is in general easy to find a differential equation for the derivatives of an arbitrary rational function, and hence also for its denominators. We sketch the procedure.

\begin{proposition} For an arbitrary rational function $Q=R/P$, the $Q^{(n)}$ satisfy a differential equation $$
P(n,z,\mathrm{d}/\mathrm{d}z)Q^{(n)}=0,$$ of order at most $\deg R+1$.
\end{proposition}
\begin{proof} Let $D$ denote $\diff/\diff z$, considered as a differential operator. A polynomial $P(z)$ can also be thought of as such an operator, acting by multiplication. The algebra of operators on polynomials, that $D$ and $z$ generate,  is the Weyl algebra $A_1$ (see \cite{Couth}). If $P$ is of degree $d$, then as operators in $A_1$, 
\begin{equation}
\label{eq:commuting}
D^nP(z)=\sum_{i=0}^{d+1}P_i(n, z)D^{n-i},
\end{equation}
for some polynomials $P_i$ in $n$ and $z$. This follows by induction on the degree of $P$, using the relation $Dz=zD+1$.

If $\deg R=d$, then $(D^nD^{d+1}P)*Q=0$ for all $n\geq 0$ (where $*$ signifies action of an operator on a function). Hence, using  (\ref{eq:commuting})
$$
\sum_{i=0}^{d+1}P_i(n+d+1, z)D^{n+d+1-i}*Q=0,$$
which means that $Q^{(n)}$ satisfies the parameter dependent differential equation
$$
\sum_{i=0}^{d+1}P_i(n+d+1, z)\frac{\diff^{d+1-i}Q^{(n)}}{\diff z^{d+1-i}}=0.
$$
\end{proof}

Using the fact that $P^nQ^{(n)}=R_n$, and a result similar to (\ref{eq:commuting}) that allows commuting high powers of $P$ with powers of $D$, it is an easy, but tedious computation to see that there is a differential equation for $R_n$ as well. We content ourselves with an easy example.

\begin{example}
\label{ex:counterexample}{\rm The following incidentally gives a counterexample to Conjecture 1 of \cite{ShSo}, described below. Let $$
Q=(z-z_1)^{-1}+(z-z_2)^{-1}=\frac{2z-z_1-z_2}{(z-z_1)(z-z_2)}.
$$ 

Then
$$
\frac{(-1)^n}{n!}Q^{(n)}=(z-z_1)^{-1-n}+(z-z_2)^{-1-n}=\frac{(z-z_1)^n+(z-z_2)^n}{(z-z_1)^n(z-z_2)^n}.
$$ 

The differential equations for $Q^{(n)}$ and $R_n$ are, respectively,
$$
\frac{(z-z_1)(z-z_2)}{(n+1)(n+2)}{Q^{(n+2)}}+\frac{2z-z_1-z_2}{n+1}{Q^{(n+1)}}+Q^{(n)}=0
$$
and
\begin{equation}
\label{eq:exampleRn}
\frac{(z-z_1)(z-z_2)}{(n+1)(n+2)}R''_n-\frac{2z-z_1-z_2}{n+1}R'_n+R_n=0.
\end{equation}

We can formally recover the algebraic equation for the Cauchy transform of the asymptotic measure from equation (\ref{eq:exampleRn}). 
Divide the equation by $n^2R_n$, let $n\to \infty$ and assume that $C:=\lim_{n\to\infty} \frac{R'_n}{nR_n}$, so that $C$ is the Cauchy transform of the asymptotic zero counting measure of $R_n$. If we further assume that $C^2=\lim_{n\to\infty} \frac{R''_n}{n^2R_n}$ (for a motivation for this in a similar situation, see
\cite{BoSh1}) we get the equation
$$
(z-z_1)(z-z_2)C^2-(2z-z_1-z_2)C+1=0.
$$
This is, of course, the equation $((z-z_1)C-1)((z-z_2)C-1)=0$, which we already established in Corollary \ref{cor:2}.

By Proposition \ref{lemma:zeroes2} there is in general no compact set that contains all the zeroes of all $R_n$.
In \cite{BoSh1} a condition is given for quasi exactly solvable parameter dependent differential equations, such as the one for $R_n$, to have polynomial solutions whose  zeroes are uniformly bounded. This condition is formulated in terms of the zeroes of an associated characteristic polynomial: these zeroes should all have different arguments.
The characteristic polynomial of (\ref{eq:exampleRn}) is $\lambda^2-2\lambda+1$ which has a double zero at $\lambda=1$. 
Hence this example gives an instance of the necessity of the condition for uniform boundedness given in \cite{BoSh1}, and in addition a counterexample to Conjecture 1 of \cite{ShSo},  which suggested that the condition in \cite{BoSh1} would be unnecessary.}
\end{example}


\section{Voronoi diagrams and P\'olya's theorem in higher dimensions}

It is tempting to try to extend the above results to higher dimensions. Without striving for generality, we will now sketch in a simple example how one could do this.  Let $W_i\subset \bC^m,\ i=1,\dotsc,d$ be a hyperplane arrangement; for simplicity of proof we will assume that 
\begin{equation}
\label{eq:generichyperplane}
W_i\cap\ W_j\neq \emptyset, \ i,j=1,\dotsc,d.
\end{equation}
For instance, this is true for all generic hyperplane arrangements, see \cite{Oda}. But for $m=1$, the only such arrangements consist of just one point, so the hypothesis is quite restrictive.

Each $W_i$ is defined by a polynomial 
 $$
 L_i(z)=\sum_{i=1}^ma_{ij}z_j+ b_i=\mathbf{a_i}\cdot z+b_i=0,\  z=(z_1,\dotsc,z_m)\in  \bC^m.$$
We make a preliminary assumption that $\vert \mathbf{a_i}\vert =1,\ i=1,\dotsc,d$, and as a consequence,
 the distance from $z$ to $W_i$ is given by $\vert L_i(z)\vert$.  Note that $\mathbf{a_i}$ is only determined up to multiplication by an element in $S^1\subset \mathbb C$, and that by (\ref{eq:generichyperplane}) all $\mathbf{a_i}$ are distinct. 
 
There is an obvious  Voronoi stratification of $\mathbb C^m$, with cells $V_i$ defined as those $z$ for which $W_i$ is the closest hyperplane, or $\min\{ \vert L_j(z)\vert,\ j=1,\dotsc,d\}=\vert L_i(z)\vert$. The codimension 1 skeleton $\Vor_{_W}^1$ consists of the union of certain closed sets of real dimension $2m-1$, subsets of $\vert L_i(z)\vert=\vert L_j(z)\vert,\ i\neq j$. As for $m=1$,
this skeleton is easily seen to be the locus of non-differentiability of the plurisubharmonic, piecewise harmonic and continous function 
$$
\Psi(z):=\smash{\displaystyle\max_{i=1,\dotsc,d}} \{\log( |L_i(z)|^{-1})\}+\log\vert P(z)\vert,
$$ 
where $P=L_1\dotsm L_d$.  

Now let us return to P\'olya's theorem. If $m=1$, $a_i=1$, and $f(z):=\sum_{i=1}^d A_i(z+b_i)^{-1},\ 0\neq A_i\in\mathbb C$, we have that
$$
f^{(n)}(z)=(-1)^nn!\sum_{i=1}^d A_i(z+b_i)^{-1-n}.
$$
Hence the zeroes of $f^{(n)}(z)$, whose asymptotic behaviour is described by P\'olya's theorem,  coincide with the zeroes of $\sum_{i=1}^d A_i(z+b_i)^{-1-n}.$ For $m>1$ this suggests one should naively consider the sequence of rational functions, $L_1^{-n}+\dotsb+L_d^{-n}$, or rather its numerator $P^n(L_1^{-n}+\dotsb+L_d^{-n})$. 

The hypersurface $Z_n:=Z(P^n(L_1^{-n}+\dotsb+L_d^{-n}))$ has an associated Euler-Poincar\'e current 
$$E_n=dd^c (\log\vert P^n(L_1^{-n}+\dotsb+L_d^{-n})\vert)$$ (see Demailly's book \cite{De}), which is the appropriate generalization of the zero-counting measure (up to a multiplicative constant). 
There is also the current $E=dd^c \Psi(z) $, with support on the the codimension 1 skeleton of the Voronoi diagram. It may be explicitly computed in a similar way as for $n=1$ above. (A sample computation is given in the example below.) 
Note that
\begin{equation}
\label{eq:uniformconv}
\Psi_n:=\frac{1}{n}\log \vert P^n( L_1^{-n}+\dotsb+L_d^{-n})\vert \to- \log\vert L_i \vert + \log\vert P\vert=\Psi
\end{equation}

in the interior of the cell that contains $W_i, \ i=1,\dotsc,d$, uniformly on compact subsets. In particular $\Psi_n$ converges to $\Psi$ pointwise a.e.
\begin{theorem}
\label{theorem:currents} Assume that $L_i(z),\ i=1,\dotsc,d$, define a generic hyperplane arrangement. Then $\frac{1}{n}E_n\to E$ as currents.
\end{theorem}
\begin{proof}It is enough to check that we have $L^1$-convergence $\Psi_n \to \Psi$ in each compact $K\subset \mathbb C^n$. First we will find a coordinate system that is adapted to use of Fubini's theorem and (the proof of) Theorem 1. We claim that it is possible to find a unit vector $\mathbf{b}$, such that all $\vert \mathbf{a_i}\cdot \mathbf{b}\vert,\ i=1,\dotsc,d,$ are
non-zero and distinct. This follows from the fact that the locus of $\vert \mathbf{a_i}\cdot \mathbf{c}\vert=\vert \mathbf{a_j}\cdot \mathbf{c}\vert$ is a codimension 1 real hypersurface if  
$\mathbf{a_i}\neq \mathbf{a_j}$. The finite union $V^h$ of these hypersurfaces cannot cover $\mathbb C^m$, and any  unit vector  $\mathbf{b}$ in the open set $ \mathbb C^m\setminus V^h$ will do. Since this last set is open, we may also assume that all $\mathbf{a_i}\cdot \mathbf{b}\neq 0$. Changing the basis of $\mathbb C^m$, so that $\mathbf{b}$ is the first basis vector, the new coordinate system then satisfies that $W_i$ is transversal to the fibres of $\rho: \rho(z_1,\dotsc,z_n)=(0,z_2,\dotsc,z_n)$, and that $\vert a_{i1}\vert,\ i=1,\dotsc,d,$ in the (new) expression for $L_i$ are all distinct. As a consequence, for each fixed $k$, and all $n$,  \begin{equation}
\label{eq:cpct}
Z_n\cap \rho^{-1}(\rho(k))\subset D_R\subset \rho^{-1}(\rho(k)),
\end{equation}
 where the radius $R$ of the disk $D_R$ with center at $\rho(k)$ may be taken to depend continuously on $\rho(k)$. Hence if $k\in K$, a compact set, $R$ may be assumed to be uniformly bounded. This is seen as follows. Assume that $\vert a_{11}\vert $ is minimal among $\vert a_{i1}\vert, \ i=1,\dotsc,d$. The restriction of $L_i$ to $\rho^{-1}(\rho(k))$ is of the form
 $L_i(z)=a_{i1}z_1+p_i$, where $p_i=p_i(z_2,\dotsc,z_m)=a_{i2}z_2+\dotsb+a_{im}z_m+b_i$ only depends on $\rho(z)$. Choose $0<\delta<1$ such that $\vert a_{i1}\vert < \delta \vert a_{s1}\vert, \ s=2,\dotsc,m$.
 If $p_1,\dotsc,p_m$ are fixed there is an $R_1$, depending continuously on $p_1,\dotsc,p_m$, such that $\vert z_1\vert> R_1$ implies that $\vert a_{i1}z_1+p_i\vert < \delta \vert a_{s1}z_1+p_s\vert, \ s=2,\dotsc,m$. Now choose an $N$ such that $n>N$ implies that
 $(d-1)\delta^n<1/2$.
Then
 $$
 \vert P(z)^n(L_1(z)^{-n}+\dotsb+L_d(z)^{-n})\vert \geq \frac{\vert P(z)\vert^n}{\vert L_1(z)\vert^{n}}(1-(d-1)\delta^n)\geq (1/2)\frac{\vert P(z)\vert^n}{\vert L_1(z)\vert^{n}}.
 $$
If the left-hand side of this inequality is zero, then $P(z)=0\iff z=-p_i/a_{i1}$ for some $i$. Hence (\ref{eq:cpct}) holds with 
$R= \max\{R_1, \vert p_i/a_{i1}\vert,\ i=1,\dotsc,d\}$ and $n>N$. To obtain (\ref{eq:cpct}) was the reason we assumed (\ref{eq:generichyperplane}).

Now we proceed to the proof proper. Fix $\epsilon>0$. Let $U$ be an open set that does not intersect the Voronoi skeleton $\Vor_L^1$, and such that $ \bar U $ is compact. By the uniform convergence (\ref{eq:uniformconv}), in compact subsets of open cells, there is an $N_1$ such that $n\geq N_1$ implies that
$\vert \Psi_n(z)-\Psi(z)\vert\leq \epsilon$ if $z\in U$. This is possible since the fibres of $\rho$ are transversal to all the hyperplanes $W_i$. Hence 
\begin{equation}
\label{eq: integral1}
I_1:=\int_{U}\vert \Psi_n(z)-\Psi(z)\vert \mathrm{d}\lambda \leq \lambda(U)\epsilon=O(\epsilon).
\end{equation}

Choose the open set 
$U\subset \bar U\subset K\setminus V^1_L$, such that for each fibre:
$\lambda_{1}(\rho^{-1}\rho(k)\cap K \setminus U)\leq \epsilon$ for all $k\in K$, where $\lambda_1$ is Lebesgue measure on $\mathbb C$. Let $N_1$ and $R$ be such that (\ref{eq:cpct}) also holds for $n>N_1$ and $k\in K$. Now 
$$
\int_{K\setminus U}\vert\Psi_n-\Psi\vert \diff \lambda\leq \int_{K\setminus U}\vert\Psi\vert \diff \lambda+\int_{K\setminus U}\vert\Psi_n\vert \diff \lambda=:I_2+I_3.
$$
Clearly by Fubini's theorem,
\begin{equation}
\label{eq:integral2}
I_2=\int_{K\setminus U}\vert\Psi\vert \diff \lambda = \int_{\rho(K)}\left(\int_{\rho^{-1}\rho(k)\cap (K\setminus U)}\vert\Psi\vert \diff\lambda_1\right)\diff\lambda_{m-1}\leq \epsilon \lambda_{m-1}(\rho(K)) M=O(\epsilon),
\end{equation}
where $M$ is the maximum value of the continous function $\Psi$ on $K$, and $\lambda_1$ and $\lambda_{m-1}$ are Lebesgue measure on $\bC$ and $\bC^{m-1},$ respectively. 

Finally, we apply Fubini again, on the integral $I_3$:
\begin{equation}
\label{eq:integral3}
I_3=\int_{K\setminus U}\vert\Psi_n\vert \diff \lambda=\int_{\rho(K)}\left(\int_{\rho^{-1}\rho(k)\cap (K\setminus U)}\vert\Psi_n\vert \diff\lambda_1\right)\diff\lambda_{m-1}.
\end{equation}
Note that for the restriction to the fiber in the inner integral 
$$R_n(z):=P(z)^n(L_1(z)^{-n}+\dotsb+L_d(z)^{-n})=A(n)z_1^{(d-1)n}+\text{ lower terms in } z_1,$$ where 
$$A(n)=\sum_{i=1}^d \frac{(a_1\cdots a_d)^n}{a_i^n}\sim \frac{(a_1\cdots a_d)^n}{a_1^n},$$
as $n\to\infty$. Hence 
$$\Psi_n(z)=(1/n)\left(\log\vert A(n)\vert+\sum_{j=1}^{n(d-1)}\log\vert z_1-\alpha_j\vert\right), $$
denoting by $\alpha_j,\ j=1,\dotsc,n(d-1)$ the roots of $R_n(z)$ on the fibre.
Each of the terms in this sum contributes to the inner integral in (\ref{eq:integral3}). Let $K\subset B_S(0)$, the ball with radius $S$ centered at the origin . By  (\ref{eq:cpct}),  $\vert \alpha_j\vert \leq R$, so that $\vert z_1-\alpha_j\vert \leq R+S$. Hence, uniformly in $\rho(k), \ k\in K$,
\begin{eqnarray*}
&& \int_{\rho^{-1}\rho(k)\cap(K\setminus U)}(1/n)\log\vert z_1-\alpha_j\vert\diff\lambda_1\\
&&=(1/n)(-(\log\epsilon-1/2)(\epsilon^2/2)+
\max\{-\log\epsilon,\log(R+S)\}\epsilon)=o(1).
\end{eqnarray*}
Thus $I_3=o(1),$ if $n$ is large enough. Adding the estimates for the three integrals $I_1,I_2$ and $I_3$ concludes the proof.

 \end{proof}

\begin{example} Consider $Q_n=z_1^{-n}-z_2^{-n}$. The set $V_{12}=\{ (z_1,z_2) : \vert z_1\vert=\vert z_2\vert\} $ is the cone over the torus $S^1\times S^1\subset \mathbb C^2$. It divides $\mathbb C^2$ into two cells. $Z_n=Z(z_1^{-n}-z_2^{-n})$ is the union of the lines $L(\epsilon_n^j)\ :\ z_1=\epsilon_n^jz_2, \ j=0,\dotsc,n-1$, where $\epsilon_n$ is a primitive $n$'th root of unity. As $n$ becomes large, it is easily seen that these lines fill out $V_{12}$.  Namely, if $(z_1,z_2)\in V_{12}$  then this point belongs to a line $L(e^{i\theta})\ :\ z_1=e^{i\theta}z_2,\ \theta \in \mathbb R$, and we can find an $\epsilon_n^j$ that is arbitrarily close to $e^{i\theta}$. Hence the line $z_1=\epsilon_n^jz_2$ will approximate the line $z_1=e^{i\theta}z_2$ (and in a compact region of $\mathbb C^2$ this approximation is uniform). 

By definition $E$ and $E_n$ acting on a $(1,1)$ test form $\alpha$ are given by 
$$E=\frac{1}{2\pi}\int_0^{2\pi}\left(\int_{ L(e^{i\theta}) }\alpha\right)\diff\theta, \hskip 1cm E_n=(1/n)\sum_{i=0}^{n-1}\int_{ L\left(\epsilon_n^j\right) }\alpha.$$

 A very similar example to $E_n\to E$ was considered by Demailly \cite{Dem82}. The current $E$ is an example of a tropical current, as defined in \cite{Babaee}. However most of the cases covered by the theorem are not tropical. 

\end{example}

\section{Some variants and further problems in complex dimension 1}
\subsection{Lemniscates}  
We will sketch another example, in the same circle of ideas as the previous, but which is not covered by P\'olya's theorem.

\begin{figure}
\begin{center}
\includegraphics[scale=0.47]{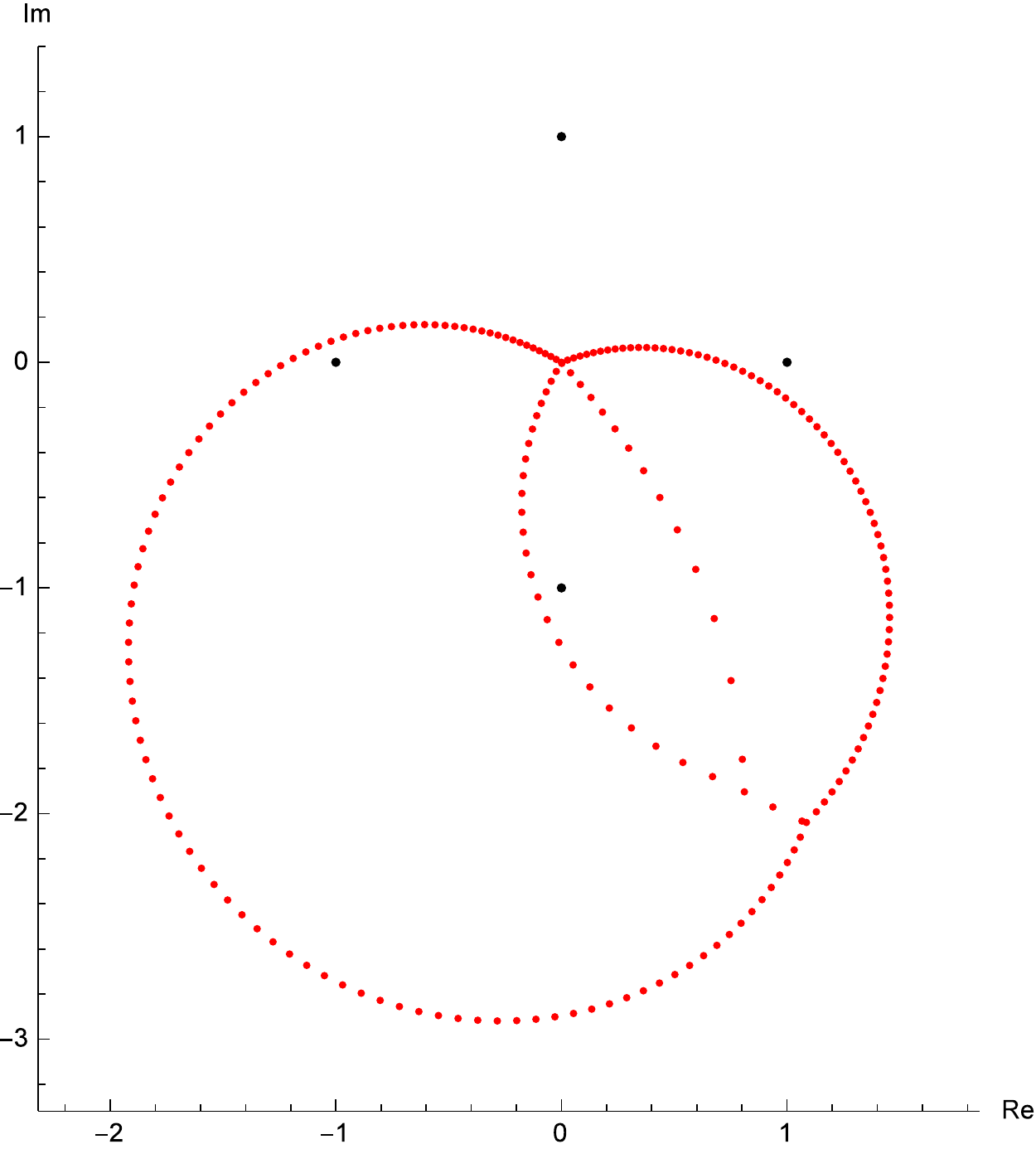}\quad \quad \includegraphics [scale=0.47]{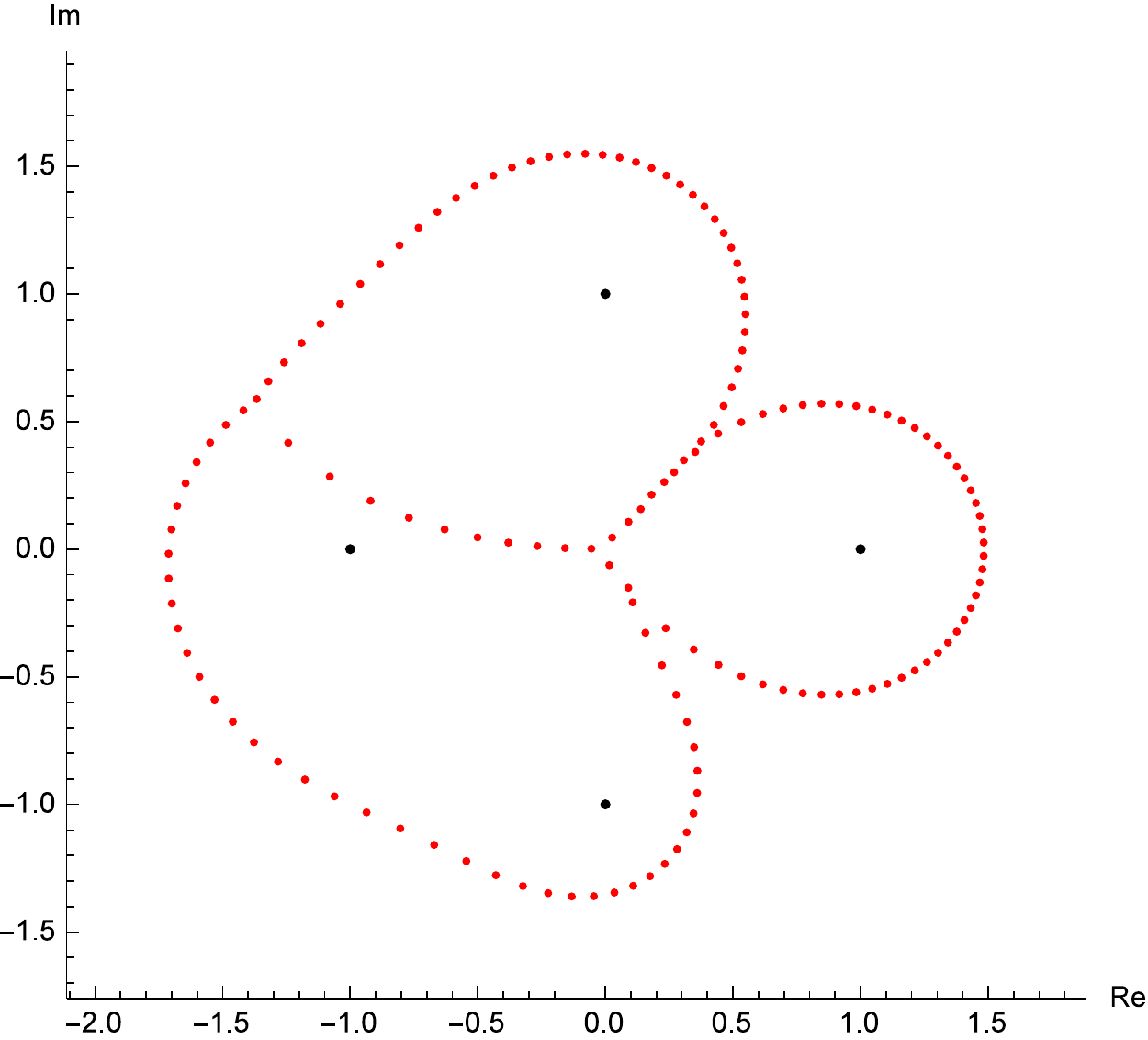} 
\end{center}
\caption{The zeroes of $(z-1)^{12n}+(z+1)^{8n}+(z-i)^{7n}+(z+i)^{21n}$ (left), and $(z-1)^{-5n}+(z+1)^{4n}+(z-i)^{-4n}+(z+i)^{-2n}$ (right), where $n=10$.}
\label{fig2}
\end{figure}
Given a finite number of entire functions $f_1,\dotsc,f_d$,  define $\Psi(z):=\max \{ \vert f_i\vert,i=1,\dotsc,d\}$. Then there is a Voronoi-type stratification with open cells of the form $V_i=\{z : \Psi(z)=\vert f_i(z)\vert\}$, with boundaries that are unions of subsets of the sets $\{z: |f_i|=|f_j|\},\,i\neq j$. Let $R_n(z)=f_1(z)^n+\dotsb+f_d(z)^n$. In the interior of $V_i$, $|f_i(z)|>|f_j(z)|$, and hence
$$
\lim_{n\to\infty}\frac{1}{n} \log\vert R_n(z)\vert = \log \vert f_i(z)\vert = \Psi(z).
$$
As a consequence this sequence is  pointwise convergent a.e.

 Now assume that $f_i(z)=P_i(z),\ i=1,\dotsc,d,$ are monic polynomials.
\begin{proposition}
\label{prop:zeroConv} Let $\mu_n$ be the zero-counting probability measure associated with $R_n=\sum_{i=1}^d P_i(z)^n$.
Then $L(\mu_n)\to \Psi$ pointwise a.e., where $$\Psi:=\max_{i=1,\dotsc,d}\{\log \vert P_i\vert\}.$$
 If $d_i:=\deg P_i$, and $d_1>d_i,\ i\neq 1,$ then the Voronoi diagram is compact, and this convergence is $L^1_{loc}$.
\end{proposition}

\begin{proof} Only the last statement needs an argument. Compactness follows by assuming that there are sequences $z_i,\ i=1,\dotsc$ and $n_i,\ i=1,\dotsc$, such that $R_{n_i}(z_i)=0$, and noting that the term $z^{d_1n}$ will dominate $R_{n_i}(z_i)$ for large $\vert z\vert$. It is furthermore easy to prove uniform convergence on compact subsets of an open Voronoi cell, and then the result follows by a similar calculation 
as in the last part of the proof of the main theorem (made easier by compactness).
\end{proof}
We expect that the convergence is still $L^1_{loc}$ without the condition on the exponents, at least in generic cases. For an illustration, see Figure ~\ref{fig2}.

\subsection{Entire functions with a finite number of zeroes} A natural extension of the preceding results is to study
$f(z)=(R/P)e^Q$ where $R,P,$ and $Q$ are polynomials, and $P$ has at least two distinct zeroes. By Hadamard's theorem functions of this type are precisely  meromorphic functions that are quotients of two entire functions of finite order, each with a finite number of zeroes.  It is immediate by induction that the number of zeroes of each $f^{(n)},\ n\geq 0,$ is finite, and hence there is a probability measure $\mu_n$ associated with $Z(f^{(n)})$. By P\'olya's theorem these zeroes cluster asymptotically along the Voronoi diagram of $P(z)$, and the precise asymptotics of $\mu_n$ seems to be given by the same measure on the Voronoi diagram as above.
More generally, leaving sequences of polynomials, one may ask for similar results for an arbitrary meromorphic function with an infinite number of zeroes. 


\begin{thebibliography}{30}

\bibitem{Babaee} F.~Babaee,  
\textit{Complex Tropical Currents, Extremality, and Approximations.} arXiv:1403.7456.

\bibitem{BR}  T.~Bergkvist and H.~Rullg\aa rd, \textit{On polynomial eigenfunctions for a class of differential operators.} Math. Res. Lett. 9 (2002).

\bibitem{BoBo} R.~B\o gvad, J.~Borcea,  \textit{Piecewise harmonic subharmonic functions and positive Cauchy transforms.} Pacific J. Math. 240 (2009), no. 2. 

\bibitem{BoSh1} R.~B\o gvad, J.~Borcea, B.~Shapiro,  
\textit{Homogenized spectral problems for exactly solvable operators: asymptotics of polynomial eigenfunctions.} Publ. Res. Inst. Math. Sci. 45 (2009).

\bibitem{BoSh} R.~B\o gvad, B.~Shapiro,  
\textit{On motherbody measures with algebraic Cauchy transform.} Enseign. Math. 62(2016), pp.117-142.

\bibitem{Couth} S. C.~Coutinho, \textit{A primer of algebraic {$D$}-modules.} Cambridge University Press, Cambridge (1995).

      
      
      
\bibitem{Dem82} J. P.~Demailly,  
\textit{Courants positifs extr\^emaux et conjecture de {H}odge.} Invent. Math. 69 (1982).



\bibitem{De} J. P.~Demailly,  
\textit{Complex analytic and differential geometry.} Available at \url{https://www-fourier.ujf-grenoble.fr/~demailly/manuscripts/agbook.pdf}

\bibitem{EbKh}
P.~Ebenfelt,  D.~Khavinson, and H.S.~Shapiro
    \textit{Two-dimensional shapes and lemniscates.} Complex analysis and dynamical systems {IV}. Part 1, Contemp. Math. 553 (2011).





\bibitem{Fu} M. {Fujiwara}, {\textit{\"Uber die obere Schranke des absoluten Betrages der Wurzeln einer algebraischen Gleichung}}, {Tohoku Math. J. 10 (1916)}, pp. {167--171}.



\bibitem{Ha} W.K.~Hayman, \textit{Meromorphic functions.} Clarendon Press, Oxford (1964).

\bibitem{HoComplex} L.~H{\"o}rmander, \textit{An introduction to complex analysis in several variables.} North-Holland Publishing Co., Amsterdam (1990).


\bibitem{Ok} A.~Okabe, B.~Boots, K.~Sugihara, and
              S.~N.~Chiu,  \textit{Spatial tessellations: concepts and applications of {V}oronoi
              diagrams.} John Wiley \& Sons, 2 ed. (2000).
              
 \bibitem{Oda} P.~Orlik and H.Terao, \textit{Arrangements of hyperplanes.} Grundlehren der Mathematischen Wissenschaften 300,
 Springer-Verlag, Berlin (1992).
             
              
              
    

\bibitem{Po1} G.~P\'olya,\textit{ \"{U}ber die {N}ullstellen sukzessiver {D}erivierten.} Math. Z., vol. 12 (1922). 

\bibitem{Po2} G.~P\'olya, \textit{On the zeroes of the derivatives of a function and its analytic character.} Bull. Amer. Math. Soc., vol. 49 (1943).
              
              \bibitem{PoRa}
S. Pouliasis and T. Ransford,
    \textit{On the harmonic measure and capacity of rational lemniscates.} Potential Anal. 44 (2016), pp. 249--261.


\bibitem{SaffTotik}E.B. Saff, V.Totik, \textit{Logarithmic potentials with external fields.} Springer-Verlag, New York (1997).

\bibitem{ShSo} B. Shapiro,  A. Solynin, \textit{Root-counting measures of Jacobi polynomials and topological types and critical geodesics of related quadratic differentials.} arXiv:1510.06003.

\bibitem{Whit} Whittaker, \textit{Interpolatory function theory.} Cambridge Tracts in Mathematics and Mathematical Physics, No. 33 (1964).

\end{thebibliography}
\end{document}